\documentclass[12pt,a4paper]{article}
\usepackage{amsmath,amsfonts,amssymb,amsthm,amscd}
\newtheorem{thm}{Theorem}
\newtheorem{lem}[thm]{Lemma}
\newtheorem{prop}[thm]{Proposition}

\begin{document}

\title{The conformal-Killing equation on $G_{2}$- and $\mathrm{Spin}_{7}$-structures}

\author{Liana David}

\maketitle

{\bf Abstract.} Let $M$ be a $7$-manifold with a $G_{2}$-structure
defined by $\phi\in \Omega^{3}_{+}(M).$ We prove that $\phi$ is
conformal-Killing with respect to the associated metric $g_{\phi}$
if and only if the $G_{2}$-structure is nearly parallel.
Similarly, let $M$ be an $8$-manifold with a
$\mathrm{Spin}_{7}$-structure defined by
$\psi\in\Omega^{4}_{+}(M).$ We prove that $\psi$ is
conformal-Killing with respect to the associated metric $g_{\psi}$
if and only if the $\mathrm{Spin}_{7}$-structure is parallel.\\


\section{Introduction}

A vector field $X$ on a Riemannian manifold $(M^{m}, g)$ is
Killing if its covariant derivative $\nabla X$ with respect to the
Levi-Civita connection $\nabla$ is totally skew-symmetric, or
$\nabla X =\frac{1}{2}dX$ (here and everywhere in this note we
identify vector fields with $1$-forms using the Riemannian
duality). More generally, a $p$-form $\psi\in\Omega^{p}(M)$ is
Killing if $\nabla \psi = \frac{1}{p+1}d\psi .$ In the same way as
Killing forms generalize Killing vector fields, conformal-Killing
forms generalize conformal-Killing vector fields.  In order to
define conformal-Killing $p$-forms ($1\leq p\leq m$) on $(M,g)$,
consider the decomposition of $T^{*}M\otimes \Lambda^{p}(M)$ into
irreducible $O(m)$-sub-bundles:
\begin{equation}\label{dec-tangent}
T^{*}M\otimes \Lambda^{p}(M)\cong \Lambda^{p+1}(M)\oplus
\Lambda^{p-1}(M)\oplus {\mathcal T}^{p}(M),
\end{equation}
where ${\mathcal T}^{p}(M)$ is the intersection of the kernels of
skew-symmetrization and natural contraction maps. The covariant
derivative $\nabla\psi$ of a $p$-form $\psi \in\Omega^{p}(M)$ is a
section of $T^{*}M\otimes \Lambda^{p}(M)$ and its projections onto
$\Lambda^{p+1}(M)$ and $\Lambda^{p-1}(M)$ according to
decomposition (\ref{dec-tangent}), are essentially given by the
exterior derivative $d\psi$ and the codifferential $\delta \psi$
respectively. The projection of $\nabla \psi$ onto the remaining
component ${\mathcal T}^{p}(M)$ defines the conformal-Killing
operator. A $p$-form $\psi$ is conformal-Killing if it belongs to
the kernel of the conformal-Killing operator, i.e. $\nabla\psi$ is
a section of the direct sum bundle
$\Lambda^{p+1}(M)\oplus\Lambda^{p-1}(M)$, or, equivalently, the
conformal-Killing equation
\begin{equation}
\nabla_{X}\psi = \frac{1}{p+1}i_{X}d\psi -\frac{1}{m-p+1}X\wedge
\delta \psi ,\quad\forall X\in TM.
\end{equation}
is satisfied.

Conformal-Killing forms exist on spaces of constant curvature, on
Sasaki manifolds and on some classes of K\"{a}hler manifolds (like
Bochner-flat or conformally Einstein) where they are closely
related to the so called Hamiltonian $2$-forms \cite{acg}.
Conformal-Killings forms exist also on Riemannian manifolds
admitting twistor spinors \cite{sem}. The space of
conformal-Killing forms on a Riemannian manifold is always finite
dimensional (even in the non-compact case) and an upper bound for
the dimension is realized on the standard sphere, where any
conformal-Killing form is a sum of two eigenforms of the Laplace
operator, with eigenvalues depending on the dimension and the
degree of the form \cite{sem}.

There are two results in the literature which motivate this note.
The first was proved in \cite{sem} and states that an almost
Hermitian manifold whose K\"{a}hler form is conformal-Killing, is
necessarily nearly K\"{a}hler; the second motivating result was
proved in \cite{liana} and states that an almost
quaternionic-Hermitian manifold whose fundamental $4$-form is
conformal-Killing, is necessarily quaternionic-K\"{a}hler. In this
note we prove the analogous statements for the fundamental form of
$G_{2}$ and $\mathrm{Spin}_{7}$-structures. More precisely, we
prove the following result:

\begin{thm}\label{main}
i) Let $M$ be a $7$-manifold with a $G_{2}$-structure defined by
$\phi\in \Omega^{3}_{+}(M)$ and let $g_{\phi}$ be the associated
Riemannian metric on $M$. Then $\phi$ is conformal-Killing with
respect to $g_{\phi}$ if and only if the $G_{2}$-structure is
nearly parallel.

ii)  Let $M$ be an $8$-manifold with a
$\mathrm{Spin}_{7}$-structure defined by
$\psi\in\Omega^{4}_{+}(M)$ and let $g_{\psi}$ be the associated
Riemannian metric on $M$. Then $\psi$ is conformal-Killing with
respect to $g_{\psi}$ if and only if the
$\mathrm{Spin}_{7}$-structure is parallel.

\end{thm}

The plan of this note is the following. In Section \ref{g2} we
recall basic facts on $G_{2}$-structures and we prove the
statement for $G_{2}$. The statement for $\mathrm{Spin}_{7}$
is proved in Section \ref{spin7}.\\

{\bf Acknowledgements.} I am grateful to Uwe Semmelmann for useful
discussions and interest in this work. Useful discussions with
Andrei Moroianu are also acknowledged. This work is partially
supported by a CNCSIS grant IDEI "Structuri geometrice pe
varietati diferentiabile" (code no. 1187/2008).

\section{The statement for $G_{2}$}\label{g2}

\subsection{Basic facts about $G_{2}$-structures}\label{g2-basic}
Let $\{ e_{1}, \cdots , e_{7}\}$ be the standard basis of $V=
\mathbb{R}^{7}$ and $\{ e^{1}, \cdots , e^{7}\}$ the dual basis.
We shall use the notation $e^{i_{1}\cdots i_{k}}$ for the wedge
product $e^{i_{1}}\wedge\cdots \wedge
e^{i_{k}}\in\Lambda^{k}(V^{*})$. Recall that $G_{2}< GL(V)$ is a
$14$-dimensional compact, connected, simple Lie group defined as
the stabilizer of the $3$-form
\begin{equation}\label{phi}
\phi_{0} := e^{123} +e^{145} + e^{167}+ e^{246} - e^{257} -e^{347}
- e^{356}.
\end{equation}
It can be shown any $g\in G_{2}$ preserves the standard metric
$\langle\cdot , \cdot\rangle_{7}$ and the orientation of $V$ for
which the basis $\{ e^{1},\cdots , e^{7}\} $ is orthonormal and
positive oriented. Let $*_{7}$ be the associated Hodge star
operator. It follows that any $g\in G_{2}$ stabilizes also the
$4$-form
\begin{equation}\label{starphi}
*_{7}\phi_{0}= e^{4567} + e^{2367} + e^{2345} + e^{1357} -
e^{1346} - e^{1256} - e^{1247}.
\end{equation}

The standard representation of $G_{2}$ on $V\cong V^{*}$ is
irreducible, but the representation of $G_{2}$ on higher degree
tensors is reducible, in general. For our purpose we need to know
the irreducible decompositions of the $G_{2}$-modules
$\Lambda^{2}(V^{*})$ and $S^{2}(V^{*})$, only. It is known that
$\Lambda^{2}(V^{*})$ decomposes into two $G_{2}$-irreducible
sub-representations
\begin{equation}\label{d1}
\Lambda^{2}(V^{*}) =\Lambda^{2}_{7}(V^{*})\oplus
\Lambda^{2}_{14}(V^{*})
\end{equation}
where
$$
\Lambda^{2}_{7}(V^{*}) := \{ *_{7}(\alpha\wedge *_{7}\phi_{0} ),\
\alpha\in V^{*}\} = \{ \beta\in\Lambda^{2}(V^{*}),\ 2*_{7}\beta =
\beta\wedge \phi_{0} \}
$$
is of dimension $7$ and
$$
\Lambda^{2}_{14}(V^{*}):= \{\beta\in\Lambda^{2}(V^{*}),\
*_{7}\beta =-\beta\wedge \phi_{0}\}
$$
is isomorphic to the adjoint representation and has dimension
$14.$ Similarly, $S^{2}(V^{*})$ decomposes into two
$G_{2}$-irreducible sub-representations
\begin{equation}\label{d2}
S^{2}(V^{*}) = S^{2}_{0}(V^{*})\oplus \mathbb{R}\langle\cdot ,
\cdot\rangle_{7}
\end{equation}
where $S^{2}_{0}(V^{*})$ denotes trace-less symmetric
$(2,0)$-tensors and $\mathbb{R}\langle\cdot , \cdot\rangle_{7}$ is
the $1$-dimensional representation generated by $\langle\cdot ,
\cdot\rangle_{7}.$ For proofs of these facts and more about the
representation theory of $G_{2}$, see e.g. \cite{bryant2},
\cite{andrei}.\\

Consider now a $7$-dimensional manifold $M$ with a
$G_{2}$-structure, defined by $\phi\in\Omega^{3}_{+}(M) .$ This
means that $\phi$ is a smooth $3$-form, linearly equivalent to
$\phi_{0}$ (i.e. at any point $p\in M$, there is a linear
isomorphism $f_{p}: T_{p}M\rightarrow V$ such that
$f_{p}^{*}(\phi_{0} ) =\phi_{p}$). The stabilizer $G_{2}\subset
GL(V)$ of $\phi_{0}$ being included in $SO(V)$, the standard
metric and orientation of $V$ induce, by means of the isomorphisms
$f_{p}$, a well-defined metric $g_{\phi}$ and an orientation on
$M$, such that $f_{p}$ is an orientation preserving isometry. We
denote by $*_{\phi}$ the associated Hodge star operator and we
freely identify vectors and covectors on $M$ using $g_{\phi}.$ Let
$\nabla$ be the Levi-Civita connection of $g_{\phi}.$ According to
\cite{sal}, \cite{gray} the covariant derivative $\nabla\phi$,
which is a section of $T^{*}M\otimes \Lambda^{3}(M)$, is actually
a section of $T^{*}M\otimes \Lambda^{3}_{7}(M)$, where
$\Lambda^{3}_{7}(M)$ is a sub-bundle of $\Lambda^{3}(M)$, defined
as
$$
\Lambda^{3}_{7}(M)= \{ *_{\phi} (\alpha\wedge \phi ),\ \alpha\in
T^{*}M\} .
$$
We shall identify $\Lambda^{3}_{7}(M)$ with $T^{*}M$ by means of
the isomorphism
\begin{equation}\label{identificare}
\Lambda^{3}_{7}(M)\ni \beta \rightarrow *_{\phi} (\beta\wedge \phi
)\in T^{*}M,
\end{equation}
which is the inverse (up to a multiplicative constant) of the
isomorphism
$$
T^{*}M\ni \alpha\rightarrow *_{\phi} (\alpha\wedge \phi )\in
\Lambda^{3}_{7}(M).
$$
More precisely, the following general identity holds:
\begin{equation}\label{general}
*_{\phi}( *_{\phi}(\alpha \wedge\phi )\wedge \phi ) =
-4\alpha,\quad\forall \alpha\in T^{*}M.
\end{equation}
Finally, from the isomorphism (\ref{identificare}) and the
decompositions (\ref{d1}) and (\ref{d2}), we get the following
decomposition of $T^{*}M\otimes \Lambda^{3}_{7}(M)$ into
irreducible sub-bundles:
\begin{equation}\label{var1}
T^{*}M\otimes \Lambda^{3}_{7}(M) \cong T^{*}M\otimes T^{*}M\cong
\Lambda^{2}_{7}(M)\oplus \Lambda^{2}_{14}(M)\oplus
S^{2}_{0}(M)\oplus\mathbb{R}_{\phi}(M),
\end{equation}
where
\begin{align*}
\Lambda^{2}_{7}(M) &= \{ \alpha\in \Lambda^{2}(M):\
2*_{\phi}\alpha =\alpha\wedge \phi \}\\
\Lambda^{2}_{14}(M) &= \{ \alpha\in \Lambda^{2}(M):\
*_{\phi}\alpha =- \alpha\wedge \phi \},
\end{align*}
$S^{2}_{0}(M)$ denotes the bundle of symmetric $(2,0)$-traceless
tensors and $\mathbb{R}_{\phi}(M)= \langle g_{\phi}\rangle$ is the
trivial rank one bundle generated by $g_{\phi }.$

From (\ref{var1}), there are 16 classes of $G_{2}$-manifolds in
the Gray-Hervella classification \cite{gray}. The
$G_{2}$-structure defined by $\phi$ is called {\it nearly
parallel} if $\nabla\phi$ is a section of the line bundle
$\mathbb{R}_{\phi}(M).$ When viewed as a sub-bundle of
$T^{*}M\otimes \Lambda^{3}_{7}(M)$, $\mathbb{R}_{\phi }(M)$ is
generated by $\sum_{k=1}^{7}e^{k}\otimes *_{\phi}(e^{k}\wedge\phi
)$, where $\{ e^{1}, \cdots , e^{7}\}$ is a local orthonormal
basis of $T^{*}M.$ Thus $\nabla\phi$ is a section of
$\mathbb{R}_{\phi}(M)$ if and only if $\nabla\phi$ is a multiple
of $*_{\phi}\phi .$

\subsection{Proof for the $G_{2}$ statement}

In this Section we prove Theorem \ref{main} {\it i)}. We use a
representation theoretic argument. Similar arguments already
appear in the literature \cite{liana}, \cite{swann2},
\cite{swann3}. Let $M$ be a $7$-manifold with a $G_{2}$-structure
defined by $\phi\in\Omega^{3}_{+}(M).$ With the notations from the
previous Section,  we aim to prove the following Proposition.

\begin{prop}\label{proof-g2} The $3$-form $\phi$ is conformal-Killing
with respect to $g_{\phi}$ if and only if the $G_{2}$-structure
defined by $\phi$ is nearly parallel.
\end{prop}

In order to prove Proposition \ref{proof-g2}, define the algebraic
conformal-Killing operator
$$
{\mathcal T}_{3}: T^{*}M\otimes \Lambda^{3}_{7}(M)\rightarrow
T^{*}M\otimes \Lambda^{3}(M)
$$
given on decomposable tensors by
\begin{equation}\label{p-1}
{\mathcal T}_{3}(\gamma\otimes\beta )(X) =\frac{3}{4}\gamma (X)
\beta +\frac{1}{4}\gamma\wedge i_{X}\beta -\frac{1}{5} X\wedge
i_{\gamma}\beta
\end{equation}
where $\gamma\in T^{*}M$, $\beta\in \Lambda^{3}_{7}(M)$, $X\in TM$
is identified with the dual $1$-form and $i_{\gamma}\beta :=\beta
(\gamma^{\flat}, \cdot )$ is the interior product of $\beta$ with
the vector field $\gamma^{\flat}$ dual to $\gamma$. (The operator
${\mathcal T}_{3}$ usually acts on the entire $T^{*}M\otimes
\Lambda^{3}(M)$, but we consider its restriction to $T^{*}M\otimes
\Lambda^{3}_{7}(M)$ only, because the covariant derivative
$\nabla\phi$ with respect to the Levi-Civita connection $\nabla$
of $g_{\phi}$ is a section of this bundle). Notice that
$$
{\mathcal T}_{3}(\nabla\phi )(X)= \nabla_{X}\phi
-\frac{1}{4}i_{X}d\phi +\frac{1}{5}X\wedge\delta \phi
,\quad\forall X\in TM.
$$
Thus $\phi$ is conformal-Killing if and only if
\begin{equation}\label{equa}
{\mathcal T}_{3}(\nabla \phi )=0.
\end{equation}
Recall now the decomposition of $T^{*}M\otimes \Lambda^{3}_{7}(M)$
into irreducible sub-bundles:
\begin{equation}\label{prime}
T^{*}M\otimes \Lambda^{3}_{7}(M)\cong \Lambda^{2}_{7}(M)\oplus
\Lambda^{2}_{14}(M)\oplus S^{2}_{0}(M)\oplus\mathbb{R}_{\phi}(M).
\end{equation}
From (\ref{equa}), Proposition \ref{proof-g2} is a consequence of
the following general result.

\begin{prop}\label{prime-prop} The restriction of ${\mathcal T}_{3}$ to $\Lambda^{2}(M)\oplus
S^{2}_{0}(M)$ is injective and the restriction of ${\mathcal
T}_{3}$ to $\mathbb{R}_{\phi}(M)$ is identically zero.
\end{prop}

We divide the proof of Proposition \ref{prime-prop} into two
Lemmas. First we need to introduce some notations. Define a map
\begin{equation}\label{extinsa}
T^{*}M\otimes \Lambda^{3}(M)\rightarrow T^{*}M\otimes T^{*}M,\quad
\alpha\otimes\beta \rightarrow \alpha\otimes
*_{\phi}(\beta\wedge\phi ).
\end{equation}
For any component $W$ in the decomposition (\ref{var1}) of
$T^{*}M\otimes T^{*}M$, let
$$
\mathrm{pr}_{W} :T^{*}M\otimes \Lambda^{3}(M)\rightarrow W
$$
be the composition of the map (\ref{extinsa}) with the projection
from $T^{*}M\otimes T^{*}M$ to $W$, according to the decomposition
(\ref{var1}). Let $\{ e^{1}, \cdots, e^{7}\}$ be a local
orthonormal positive oriented frame of $T^{*}M$, so that $\phi$ is
of the form (\ref{phi}), and define
$$
\eta := e^{1}\otimes *_{\phi}(e^{2}\wedge\phi ).
$$
From its very definition, $\eta$ is a section of $T^{*}M\otimes
\Lambda^{3}_{7}(M)$.

\begin{lem} With the notations above,
\begin{equation}\label{la-1}
(\mathrm{pr}_{\Lambda^{2}(M)}\circ{\mathcal T}_{3})(\eta ) =
-\frac{9}{10}(3e^{12} + e^{47} + e^{56})
\end{equation}
and
\begin{equation}\label{la-2}
(\mathrm{pr}_{S^{2}_{0}(M)}\circ{\mathcal T}_{3})(\eta )=
-\frac{7}{4} (e^{1}\otimes e^{2}+ e^{2}\otimes e^{1}).
\end{equation}
In particular $(\mathrm{pr}_{\Lambda^{2}_{7}(M)}\circ{\mathcal
T}_{3})(\eta )$, $(\mathrm{pr}_{\Lambda^{2}_{14}(M)}\circ{\mathcal
T}_{3})(\eta )$ and $(\mathrm{pr}_{S^{2}_{0}(M)}\circ{\mathcal
T}_{3})(\eta )$ are non-zero and ${\mathcal
T}_{3}\vert_{\Lambda^{2}(M)\oplus S^{2}_{0}(M)}$ is injective.
\end{lem}

\begin{proof} From (\ref{p1}),
$$
{\mathcal T}_{3}(\eta ) =\frac{3}{4}e^{1}\otimes
*_{\phi}(e^{2}\wedge\phi ) +\frac{1}{4} e^{k}\otimes
 e^{1}\wedge *_{\phi}(e^{k}\wedge e^{2}\wedge\phi ) -\frac{1}{5} e^{k}\otimes
e^{k}\wedge *_{\phi}(e^{1}\wedge e^{2}\wedge \phi )
$$
where, in order to simplify notations, we omitted the summation
sign over $1\leq k\leq 7$. Therefore, from the identification
(\ref{identificare}),
\begin{align*}
(\mathrm{pr}_{\Lambda^{2}(M)}\circ{\mathcal T}_{3})(\eta )&=
\frac{3}{4} e^{1}\wedge *_{\phi}(*_{\phi}(e^{2}\wedge \phi )
\wedge \phi )+\frac{1}{4} e^{k}\wedge *_{\phi} (e^{1}\wedge
*_{\phi}(e^{k}\wedge e^{2}\wedge\phi ) \wedge\phi )\\
&-\frac{1}{5} e^{k}\wedge *_{\phi} (e^{k}\wedge *_{\phi}
(e^{1}\wedge e^{2}\wedge\phi )\wedge \phi ).
\end{align*}
The first term from the right hand side of this equality is equal
to $-3e^{1}\wedge e^{2}$, using the identity (\ref{general}). For
the second term, we compute
$$
*_{\phi}(e^{2}\wedge \phi )= - (e^{367} + e^{345} + e^{156}
+e^{147})
$$
and, for any $k\in \{ 1,2,\cdots , 7\}$,
\begin{align*}
*_{\phi}( e^{k}\wedge e^{2}\wedge \phi ) &= -
\delta_{k1}(e^{47} + e^{56}) - \delta_{k3} (e^{45} + e^{67})\\
&+\delta_{k4} (e^{17} + e^{35}) + \delta_{k5} ( e^{16}- e^{34})\\
&+\delta_{k6} (e^{37}-e^{15}) - \delta_{k7} (e^{14}+ e^{36}).
\end{align*}
It follows that
$$
e^{1}\wedge \phi \wedge *_{\phi} (e^{k}\wedge e^{2}\wedge\phi )=
2\delta_{k1} e^{134567} -\delta_{k4}e^{123456} +\delta_{k5}
e^{123457} +\delta_{k6} e^{123467} -\delta_{k7} e^{123567}
$$
and
$$
*_{\phi}(e^{1}\wedge \phi \wedge *_{\phi} (e^{k}\wedge
e^{2}\wedge\phi ))=- 2\delta_{k1} e^{2} -\delta_{k4}e^{7}
-\delta_{k5} e^{6} +\delta_{k6} e^{5} +\delta_{k7} e^{4}.
$$
Therefore,
$$
e^{k}\wedge *_{\phi}(e^{1}\wedge \phi \wedge *_{\phi} (e^{k}\wedge
e^{2}\wedge\phi ))= -2(e^{12} +e^{47} +e^{56}),
$$
which provides the second term in the expression of
$(\mathrm{pr}_{\Lambda^{2}(M)}\circ{\mathcal T}_{3})(\eta ).$ A
similar calculation finally shows that
$$
e^{k}\wedge *_{\phi} ( e^{k}\wedge *_{\phi}(e^{1}\wedge
e^{2}\wedge \phi ) \wedge \phi )= 2(- 2e^{12}+e^{47} + e^{56}).
$$
Putting together the results of these computations, we obtain
(\ref{la-1}). Relation (\ref{la-2}) follows from a similar
computation. Using (\ref{la-1}) it can be checked that
\begin{equation}\label{rel1}
(\mathrm{pr}_{\Lambda^{2}(M)}\circ{\mathcal T}_{3})(\eta )\wedge
\phi = \frac{9}{5}( e^{12347} + e^{12356} + e^{34567})
\end{equation}
and
\begin{equation}\label{rel2}
*_{\phi}(\mathrm{pr}_{\Lambda^{2}(M)}\circ {\mathcal T}_{3})(\eta
) = -\frac{9}{10}( e^{12347} + e^{12356} + 3e^{34567}).
\end{equation}
Relations (\ref{rel1}) and (\ref{rel2}) imply that
$(\mathrm{pr}_{\Lambda^{2}_{7}(M)}\circ{\mathcal T}_{3})(\eta )$
and $(\mathrm{pr}_{\Lambda^{2}_{14}(M)}\circ{\mathcal T}_{3})(\eta
)$ are both non-zero. From the decomposition (\ref{prime}) and an
easy argument which uses the Schur's lemma,
$\mathrm{pr}_{\Lambda^{2}_{7}(M)}\circ{\mathcal T}_{3}$,
$\mathrm{pr}_{\Lambda^{2}_{14}(M)}\circ{\mathcal T}_{3}$ and
$\mathrm{pr}_{S^{2}_{0}(M)}\circ{\mathcal T}_{3}$ non-trivial
imply that ${\mathcal T}\vert_{\Lambda^{2}(M)\oplus S^{2}_{0}(M)}$
is injective, as required.
\end{proof}

In order to conclude the proof of Proposition \ref{prime-prop} we
still need to show that $\mathcal T$ is identically zero on
$\mathbb{R}_{\phi}(M)$. This is done in the following Lemma.

\begin{lem} The restriction ${\mathcal T}_{3}\vert_{\mathbb{R}_{\phi}(M)}$
is identically zero.
\end{lem}

\begin{proof} Recall that $T^{*}M\otimes \Lambda^{3}(M)$
contains $\mathbb{R}_{\phi }(M)$ with multiplicity one and the
projection $\mathrm{pr}_{\mathbb{R}_{\phi}(M)}$ onto
$\mathbb{R}_{\phi}(M)$ has the following expression:
\begin{equation}\label{p1}
T^{*}M\otimes\Lambda^{3}(M)\ni\gamma\otimes\beta \rightarrow
*_{\phi}(\gamma\wedge \beta\wedge \phi
)g_{\phi}\in\mathbb{R}_{\phi}(M).
\end{equation}
Schur's Lemma again implies that ${\mathcal
T}_{3}\vert_{\mathbb{R}_{\phi}(M)}$ is identically zero if and
only if $\mathrm{pr}_{\mathbb{R}_{\phi}(M)}\circ{\mathcal T}_{3}$
is identically zero. On the other hand, from definition
(\ref{p-1}) of ${\mathcal T}_{3}$ and (\ref{p1}),
\begin{align*}
(\mathrm{pr}_{\mathbb{R}_{\phi}(M)}\circ {\mathcal T}_{3})(\gamma
\otimes\beta )&= \mathrm{pr}_{\mathbb{R}_{\phi}(M)}\left(
e^{k}\otimes {\mathcal T}_{3}(\gamma\otimes\beta )(e_{k})\right)\\
&= \mathrm{pr}_{\mathbb{R}_{\phi}(M)} \left( \frac{3}{4}
\gamma\otimes\beta +\frac{1}{4} e^{k}\otimes \gamma\wedge
i_{e_{k}}\beta -\frac{1}{5} e^{k}\otimes e^{k}\wedge
i_{\gamma}\beta \right)\\
&= *_{\phi}\left(\left( \frac{3}{4}\gamma\wedge \beta +\frac{1}{4}
e^{k}\wedge \gamma\wedge i_{e_{k}}\beta \right)\wedge\phi\right) g_{\phi}\\
\end{align*}
which is zero, since
$$
3\gamma\wedge\beta +e^{k}\wedge \gamma\wedge i_{e_{k}}\beta
=0,\quad\forall\gamma\in T^{*}M,\quad\forall\beta\in
\Lambda^{3}(M).
$$
Our claim follows.
\end{proof}

The proof is of Proposition \ref{prime-prop} is now completed.
Thus Proposition \ref{proof-g2} and Theorem \ref{main} {\it ii)}
follow.

\section{The statement for $\mathrm{Spin}_{7}$}\label{spin7}

\subsection{Basic facts about
$\mathrm{Spin}_{7}$-structures}\label{spin7-basic}

Extend the vector space $V =\mathbb{R}^{7}$ of Section
\ref{g2-basic} to $V_{+}: = \mathbb{R}e_{0}\oplus V$. The group
$\mathrm{Spin}_{7}< GL(V_{+})$ is a $21$-dimensional compact,
connected, simply connected Lie group defined as the stabilizer of
the $4$-form
\begin{equation}\label{psi0}
\psi_{0} = e^{0}\wedge\phi_{0} + *_{7}\phi_{0} ,
\end{equation}
where we used the isomorphism $\Lambda^{4}(V_{+}^{*}) \cong
\Lambda^{3}(V^{*})\oplus \Lambda^{4}(V^{*})$, $\phi_{0}$ and
$*_{7}$ were defined in Section \ref{g2-basic} and $e^{0}\in
(V_{+})^{*}$ takes value one on $e_{0}$ and annihilates $V.$ In
terms of the standard basis $\{ e_{0}, \cdots , e_{7}\}$ of
$V_{+}$ and the dual basis $\{ e^{0}, \cdots , e^{7}\}$,
\begin{align*}
\psi_{0} =e^{0123}+e^{0145}+e^{0167}+e^{0246}-e^{0257}-e^{0347}-e^{0356}\\
+e^{4567} +e^{2367}+e^{2345}+e^{1357}-e^{1346}-e^{1256}-e^{1247}.
\end{align*}
From (\ref{psi0}), the stabilizer of $e_{0}$ in
$\mathrm{Spin}_{7}$ is isomorphic to $G_{2}.$ The vector space
$V_{+}$ has a standard metric $\langle\cdot , \cdot\rangle_{8}$
and orientation, for which $\{ e^{0}, \cdots , e^{7}\}$ is
orthonormal and positive oriented and we shall denote by $*_{8}$
the associated Hodge star operator. It can be shown that
$\mathrm{Spin}_{7}$ preserves this metric and orientation and the
map $\mathrm{Spin}_{7}\rightarrow S^{7}$ defined by $g\rightarrow
g(e_{0})$ is a $G_{2}$-fibration.

The group $\mathrm{Spin}_{7}$ acts irreducibly on $V_{+}\cong
V_{+}^{*}$, but its action on higher degree forms is reducible in
general. For our purpose we need to recall the irreducible
decomposition of $\Lambda^{3}(V_{+}^{*})$ only. As shown in
\cite{bryant1}, $\Lambda^{3}(V_{+}^{*})$ decomposes into
irreducible $\mathrm{Spin}_{7}$-modules as
\begin{equation}\label{lambda3}
\Lambda^{3}(V_{+}^{*})=\Lambda^{3}_{8}(V_{+}^{*}) \oplus
\Lambda^{3}_{48}(V_{+}^{*}),
\end{equation}
where $\Lambda^{3}_{8}(V_{+}^{*})$ and
$\Lambda^{3}_{48}(V_{+}^{*})$, of dimension $8$ and $48$
respectively, are defined by
\begin{align*}
\Lambda^{3}_{8}(V_{+}^{*}) &:= \{ *_{8} (\psi_{0}\wedge \alpha ),
\quad\alpha\in V_{+}^{*}\}\\
\Lambda^{3}_{48}(V_{+}^{*}) &:= \{ \beta\in\Lambda^{3}(V_{+}^{*}),
\quad \beta\wedge \psi_{0} =0\} .
\end{align*}

Consider now a $\mathrm{Spin}_{7}$-structure on an $8$-manifold
$M$, defined by a smooth $4$-form $\psi\in \Omega^{4}_{+}(M)$
linearly equivalent to $\psi_{0}$ at any point of $M$. Since
$\mathrm{Spin}_{7}< SO(V_{+})$, $\psi$ determines a canonical
metric $g_{\psi}$ and an orientation on $M$, for which any linear
isomorphism $f_{p}: T_{p}M\rightarrow V_{+}$ with
$f_{p}^{*}(\psi_{0})= \psi_{p}$ is an orientation preserving
isometry. We denote by $*_{\psi}$ the Hodge star operator
associated to $g_{\psi}$ and this orientation. We shall freely
identify vectors and covectors on $M$ using $g_{\psi }.$ Let
$\nabla$ be the Levi-Civita connection of $g_{\psi}.$  As shown in
\cite{sal} and \cite{gray} (see also \cite{bryant2}) the covariant
derivative $\nabla\psi$ is a section of the tensor product
$T^{*}M\otimes \Lambda^{4}_{7}(M)$, where $\Lambda^{4}_{7}(M)$ is
a rank $7$-bundle generated by $4$-forms $X\wedge i_{Y}\psi -
Y\wedge i_{X}\psi$, for any $X, Y\in TM$. Moreover,  the inner
contraction map
\begin{equation}\label{inner}
T^{*}M\otimes \Lambda^{4}_{7}(M)\rightarrow \Lambda^{3}(M),
X\otimes \alpha \rightarrow i_{X}\alpha
\end{equation}
is an isomorphism \cite{fernandez}. From (\ref{lambda3}) and
(\ref{inner}), the isomorphic bundles $T^{*}M\otimes
\Lambda^{4}_{7}(M)$ and $\Lambda^{3}(M)$ decompose into
irreducible sub-bundles as
\begin{equation}\label{ire}
T^{*}M\otimes \Lambda^{4}_{7}(M) \cong \Lambda^{3}(M)\cong
\Lambda^{3}_{8}(M)\oplus \Lambda^{3}_{48}(M),
\end{equation}
where
$$
\Lambda^{3}_{8}(M) := \{ *_{\psi} (\psi \wedge \alpha ), \quad
\alpha\in T^{*}M\}
$$
and
$$
\Lambda^{3}_{48}(M) := \{ \beta\in\Lambda^{3}(M), \quad
\beta\wedge \psi =0\} .
$$
It follows that there are four classes of
$\mathrm{Spin}_{7}$-manifolds \cite{fernandez}. The
$\mathrm{Spin}_{7}$-structure is called {\it parallel} if
$\nabla\psi =0.$

\subsection{Proof for the $\mathrm{Spin}_{7}$ statement}

In this Section we prove Theorem \ref{main} {\it ii)}. We use a
similar method like for Theorem \ref{main} {\it i)}, but the
computations are more involved. Let $M$ be an $8$-manifold with a
$\mathrm{Spin}_{7}$-structure defined by
$\psi\in\Omega^{3}_{+}(M).$ With the notations from the previous
Section, we aim to prove the following Proposition.

\begin{prop}\label{p-spin} The $4$-form $\psi$ is conformal-Killing
with respect to $g_{\psi}$ if and only if the
$\mathrm{Spin}_{7}$-structure defined by $\psi$ is parallel.
\end{prop}

Like before, consider the algebraic conformal-Killing operator
\begin{equation}\label{o1}
{\mathcal T}_{4} :T^{*}M\otimes \Lambda^{4}_{7}(M) \rightarrow
T^{*}M\otimes \Lambda^{4}(M)
\end{equation}
given on decomposable tensors by
\begin{equation}\label{o2}
{\mathcal T}_{4}(\gamma\otimes \alpha ) (X)= \frac{4}{5} \gamma
(X)\alpha +\frac{1}{5} \gamma \wedge i_{X}\alpha -\frac{1}{5}
X\wedge i_{\gamma}\alpha
\end{equation}
where $\alpha\in \Lambda^{4}_{7}(M)$, $X\in TM$ is identified with
a covector using $g_{\psi}$ and $i_{\gamma}\alpha := \alpha
(\gamma^{\flat},\cdot )$ denotes the inner product of $\alpha$ and
the dual vector field $\gamma^{\flat}$ corresponding to $\gamma
\in T^{*}M$. (${\mathcal T}_{4}$ usually acts on the entire
$T^{*}M\otimes \Lambda^{4}(M)$ but we consider its restriction to
$T^{*}M\otimes \Lambda^{4}_{7}(M)$ only, because the covariant
derivative $\nabla\psi$ with respect to the Levi-Civita connection
$\nabla$ of $g_{\psi}$ is a section of this bundle). Note that
$$
{\mathcal T}_{4}(\nabla \psi )(X) =\nabla_{X}\psi -
\frac{1}{5}i_{X}d\psi +\frac{1}{5} X\wedge \delta \psi
,\quad\forall X\in TM
$$
and thus $\psi$ is conformal-Killing if and only if
\begin{equation}\label{spin-conf}
{\mathcal T}_{4}(\nabla \psi )=0.
\end{equation}

From (\ref{spin-conf}), Proposition \ref{p-spin} is a consequence
of the following general result.

\begin{prop}\label{a-spin} The algebraic conformal-Killing operator
${\mathcal T}_{4}$ defined by (\ref{o1}) and (\ref{o2}) is
injective.\end{prop}

In order to prove Proposition \ref{a-spin}, we will find maps
$$
P_{8}: T^{*}M\otimes \Lambda^{4}(M) \rightarrow \Lambda^{3}_{8}(M)
$$
and
$$
P_{48}: T^{*}M\otimes \Lambda^{4}(M) \rightarrow
\Lambda^{3}_{48}(M)
$$
such that the compositions $P_{8}\circ {\mathcal T}_{4}$ and
$P_{48}\circ {\mathcal T}_{4}$ are non-trivial. (The existence of
such maps would readily imply, from Schur's lemma and the
decomposition (\ref{ire}) of $T^{*}M\otimes \Lambda^{4}_{7}(M)$,
that ${\mathcal T}_{4}$ is injective, as required in Proposition
\ref{a-spin}).\\

In order to define the maps $P_{8}$ and $P_{48}$ we need to
introduce more notations, as follows. First, define
$$
p:\Lambda^{4}(M) \rightarrow \Lambda^{2}(M)
$$
by
\begin{equation}\label{def-p}
p(\beta ) (X, Y) = \langle \psi , i_{X}\beta\wedge Y - i_{Y}\beta
\wedge X\rangle ,\quad \beta\in\Lambda^{4}(M),\quad X ,Y\in TM,
\end{equation}
where $\langle\cdot , \cdot\rangle$ denotes the inner product
which on decomposable skew-symmetric multi-vectors (or forms) is
defined by
$$
\langle X_{1}\wedge \cdots X_{p}, Y_{1}\wedge \cdots Y_{p}\rangle
=\mathrm{det}g_{\psi}(X_{i}, Y_{j}).
$$
Next, by tensoring the map $p$ with the identity map on $T^{*}M$
we get a map
$$
\mathrm{id}\otimes p:T^{*}M\otimes \Lambda^{4}(M)\rightarrow
T^{*}M\otimes \Lambda^{2}(M),
$$
which, composed at the right with skew symmetrization, gives a map
$$
P: T^{*}M\otimes \Lambda^{4}(M) \rightarrow \Lambda^{3}(M).
$$
Finally, composing further $P$ with the projections
$\Lambda^{3}(M)\rightarrow \Lambda^{3}_{8}(M)$ and
$\Lambda^{3}(M)\rightarrow \Lambda^{3}_{48}(M)$ according to
(\ref{ire}), we get the two maps $P_{8}$ and $P_{48}$ we were
looking for.

In order to show that the compositions $P_{8}\circ {\mathcal
T}_{4}$ and $P_{48}\circ {\mathcal T}_{4}$ are non-trivial, we
will find a particular $\eta\in T^{*}M\otimes\Lambda^{4}(M)$ such
that both $P_{8}\circ {\mathcal T}_{4}$ and $P_{48}\circ {\mathcal
T}_{4}$ take non-zero value on $\eta .$ To define $\eta$, consider
a local positive oriented orthonormal frame $\{ e_{0}, \cdots ,
e_{7}\}$ of $TM$ and its dual frame $\{ e^{0}, \cdots , e^{7}\}$,
such that $\psi$ has the form
\begin{align*}
\psi =
e^{0123}+e^{0145}+e^{0167}+e^{0246}-e^{0257}-e^{0347}-e^{0356}\\
+e^{4567} +e^{2367}+e^{2345}+e^{1357}-e^{1346}-e^{1256}-e^{1247}.
\end{align*}
Define
$$
\alpha_{0}:= i_{e_{0}}\psi \wedge e^{1}- i_{e_{1}}\psi \wedge
e^{0}
$$
and
$$
\eta := e^{0}\otimes \alpha_{0}.
$$
From its very definition, $\alpha_{0}$ is a section of
$\Lambda^{4}_{7}(M)$ and $\eta$ is a section of $T^{*}M\otimes
\Lambda^{4}_{7}(M).$ In the following Lemmas we will compute
$(P\circ {\mathcal T}_{4})(\eta )$. Since
\begin{equation}\label{pc}
(P\circ {\mathcal T}_{4})(\gamma\otimes \alpha ) =
\frac{4}{5}\gamma\wedge p(\alpha ) +\frac{1}{5} e^{k}\wedge p(
\gamma\wedge i_{e_{k}}\alpha ) - \frac{1}{5}e^{k}\wedge
p(e^{k}\wedge i_{\gamma}\alpha ) ,
\end{equation}
for any $\gamma \in T^{*}M$ and $\alpha\in \Lambda^{4}(M)$,  we
need to compute $p(\alpha_{0})$, $e^{k}\wedge p(e^{k}\wedge
i_{e_{0}}\alpha_{0})$ and $e^{k}\wedge p( e^{0}\wedge
i_{e_{k}}\alpha_{0})$. (As usual, we omit the summation sum over
$0\leq k\leq 7$).\\

First, we compute $p(\alpha_{0}).$

\begin{lem}\label{l1}
The $3$-form $p(\alpha_{0})$ has the following expression:
\begin{equation}\label{e1}
p(\alpha_{0}) = -8(e^{01}+e^{23}+e^{45}+e^{67}).
\end{equation}
\end{lem}

\begin{proof} From the definition (\ref{def-p}) of the map $p$,
$$
p(\alpha_{0})(X,Y)= \langle \psi , i_{X}\alpha_{0}\wedge Y -
i_{Y}\alpha_{0}\wedge X\rangle ,\quad\forall X, Y\in TM.
$$
Define a $1$-form $\psi (i_{X}\alpha_{0}, \cdot )$ by
$$
\psi (i_{X}\alpha_{0}, Y):= \langle \psi , i_{X}\alpha_{0}\wedge
Y\rangle ,\quad \forall Y\in TM.
$$
With this notation,
\begin{equation}\label{alpha-0}
p(\alpha_{0})(X,Y) =\psi (i_{X}\alpha_{0}, Y) -\psi
(i_{Y}\alpha_{0}, X).
\end{equation}
In terms of the local frame $\{ e^{0}, \cdots ,e^{7}\}$ chosen
above,
\begin{equation}\label{detaliat}
\alpha_{0}= -e^{1246} +e^{1257} +e^{1347} +e^{1356} + e^{0357}
-e^{0346}- e^{0256} -e^{0247}
\end{equation}
and, from a straightforward computation,
\begin{align*}
i_{X}\alpha_{0} &= e^{0}(X) (e^{357} -e^{346} - e^{256}
-e^{247}) +e^{1}(X)( -e^{246} +e^{257}+e^{347}+e^{356})\\
& +e^{2}(X) (e^{146}-e^{157}+e^{056}+e^{047})
+e^{3}(X) (-e^{147}-e^{156}-e^{057}+e^{046})\\
&+e^{4}(X)(-e^{126}+e^{137}-e^{036}-e^{027} )
+e^{5}(X)(e^{127}+e^{136}+e^{037}-e^{026})\\
&+e^{6}(X) (e^{124}-e^{135}+e^{034}+e^{025})
+e^{7}(X)(-e^{125}-e^{134}-e^{035}+e^{024}).
\end{align*}
Using this computation and the expression of $\psi$ in the frame
$\{ e_{0}, \cdots , e_{7}\}$ we get
\begin{align*}
\psi (i_{X}\alpha_{0}, \cdot )&= -4e^{0}(X) e^{1} +4e^{1}(X) e^{0}
-4e^{2}(X) e^{3}+4
e^{3}(X) e^{2}\\
& -4e^{4}(X)e^{5} +4e^{5}(X)e^{4} -4e^{6}(X) e^{7}
+4e^{7}(X)e^{6}.
\end{align*}
Applying this relation to $Y$, skew-symmetrizing the result in $X$
and $Y$ and using (\ref{alpha-0}) we get (\ref{e1}), as required.
\end{proof}

Next, we compute $e^{k}\wedge p(e^{k}\wedge i_{e_{0}}\alpha_{0})$.

\begin{lem}\label{l2} The $3$-form $e^{k}\wedge p(e^{k}\wedge
i_{e_{0}}\alpha_{0})$ has the following expression:
\begin{equation}\label{comut1}
e^{k}\wedge p(e^{k}\wedge i_{e_{0}}\alpha_{0})= -8(e^{045}
 +e^{023} +e^{067}) -6(e^{256} +e^{247}+e^{346}-e^{357}).
\end{equation}
\end{lem}

\begin{proof}
From the definition (\ref{def-p}) of the map $p$, for any $k\in \{
0,\cdots , 7\}$ and $X, Y\in TM$,
\begin{align*}
p(e^{k}\wedge i_{e_{0}}\alpha_{0})(X,Y) &= e^{k}(X)\langle \psi
,i_{e_{0}}\alpha_{0}\wedge Y\rangle - e^{k}(Y) \langle \psi ,
i_{e_{0}}\alpha_{0}\wedge X\rangle\\
& + \langle\psi , i_{X}(i_{e_{0}}\alpha_{0})\wedge Y\wedge
e_{k}\rangle -\langle\psi , i_{Y}(i_{e_{0}}\alpha_{0})\wedge X
\wedge e_{k}\rangle .
\end{align*}
From
\begin{equation}\label{i}
i_{e_{0}}\alpha_{0}= e^{357}-e^{346}-e^{256}-e^{247}
\end{equation}
and the expression of $\psi$ in the frame $\{ e_{0}, \cdots ,
e_{7}\}$, we get
\begin{equation}\label{E1}
\langle \psi , i_{e_{0}}\alpha_{0}\wedge X\rangle  =
-4e^{1}(X),\quad\forall X\in TM
\end{equation}
and thus
\begin{equation}\label{1k}
e^{k}(X)\langle \psi ,i_{e_{0}}\alpha_{0}\wedge Y\rangle -
e^{k}(Y) \langle \psi , i_{e_{0}}\alpha_{0}\wedge X\rangle =
4e^{1k}(X, Y).
\end{equation}
It remains to compute the terms of the form $\langle \psi ,
i_{X}(i_{e_{0}}\alpha_{0})\wedge Y\wedge e_{k}\rangle .$ For this,
define a $2$-form $\psi (i_{X}(i_{e_{0}}\alpha_{0}),\cdot )$ whose
value on a pair of vectors $(Y, Z)$ is equal to $\langle \psi ,
i_{X}(i_{e_{0}}\alpha_{0} )\wedge Y\wedge Z\rangle$. Taking the
inner product of $i_{e_{0}}\alpha_{0}$ given by (\ref{i}) with
$X\in TM$ and contracting the resulting expression with $\psi$ we
get
\begin{align*}
\psi (i_{X}(i_{e_{0}}\alpha ), \cdot )&= e^{5}(X)
(-2e^{04} -e^{37} +e^{26} +2e^{15})+e^{6}(X)(2e^{07} -e^{34} -e^{25} +2e^{16})\\
&+e^{4}(X)(2e^{05} +e^{36} +e^{27} +2e^{14})+e^{3}(X)(-2e^{02} -e^{46}+2e^{13} +e^{57})\\
&+e^{2}(X)(2e^{03}+2e^{12} -e^{47} -e^{56}) +e^{7}(X) (
-2e^{06}+e^{35} +2e^{17} -e^{24}).
\end{align*}
Applying $\psi (i_{X}(i_{e_{0}}\alpha ),\cdot )$  to $(Y,Z)$ and
skew-symmetrizing the result in $X$ and $Y$ we obtain
\begin{align*}
\langle \psi , i_{X}(i_{e_{0}}\alpha_{0}) \wedge Y\wedge Z -
i_{Y}(i_{e_{0}}\alpha_{0})\wedge X\wedge Z\rangle =
-4(e^{23}+e^{45} +e^{67})(X,Y)e^{0}(Z)\\
+2(e^{05}+e^{36}-e^{14}+e^{27})(X,Y) e^{4}(Z)
+2(e^{35}-e^{06}-e^{24}-e^{17})(X,Y)e^{7}(Z)\\
+2(e^{57}-e^{46}-e^{13}-e^{02})(X,Y)e^{3}(Z)
+2(e^{52}+e^{61}+e^{43}+e^{07})(X,Y)e^{6}(Z)\\
+2(e^{65}+e^{74}+e^{03}+e^{21})(X,Y)e^{2}(Z)
+2(e^{51}+e^{26}+e^{40}+e^{73})(X,Y)e^{5}(Z),\\
\end{align*}
for any $X, Y, Z\in TM.$ Letting in this expression $Z:= e^{k}$
and using the definition of $p$ together with (\ref{1k}) we get
\begin{align*}
p(e^{k}\wedge i_{e_{0}}\alpha_{0}) &= 4e^{1k}
-4\delta_{k0}(e^{23}+e^{45}+e^{67})+2\delta_{k4}(e^{05} +e^{36}
-e^{14} +e^{27})\\
&+2\delta_{k7}(e^{35} -e^{06} -e^{24} -e^{17})+2\delta_{k3}(
e^{57}-e^{46} -e^{13} -e^{02})\\
&+2\delta_{k6}(e^{52}+ e^{61}+ e^{43}+
e^{07})+2\delta_{k2}(e^{65}+e^{74}+e^{03} +e^{21})\\
& +2\delta_{k5} (e^{51}+e^{26} +e^{40} +e^{73}),
\end{align*}
for any fixed $k$. Relation (\ref{comut1}) follows now easily.
\end{proof}

Finally, it remains to compute $e^{k}\wedge p(e^{0}\wedge
i_{e_{k}}\alpha_{0}).$

\begin{lem}\label{l3} The $3$-form $e^{k}\wedge p(e^{0}\wedge
i_{e_{k}}\alpha_{0})$ has the following expression:
$$
e^{k}\wedge p(e^{0}\wedge i_{e_{k}}\alpha_{0}) =
6(e^{247}-e^{357}+e^{256}+e^{346}).
$$
\end{lem}

\begin{proof}
From the definition of the map $p$, one can check that
\begin{align*}
p(e^{0}\wedge i_{e_{k}}\alpha_{0} )(X,Y)&= 4 (-\delta_{k0}e^{01}
-\delta_{k2}e^{03}+\delta_{k3}e^{02}-\delta_{k4}e^{05})(X,Y)\\
& +4(\delta_{k5}e^{04}-\delta_{k6}e^{07}
+\delta_{k7}e^{06})(X,Y)\\
&-\langle \phi , i_{X}(i_{e_{k}}\alpha_{0})\wedge Y -
i_{Y}(i_{e_{k}}\alpha_{0})\wedge X\rangle ,
\end{align*}
where
$$
\phi := e^{123} +e^{145} + e^{167}+ e^{246} - e^{257} -e^{347} -
e^{356}
$$
and $k\in \{ 0,\cdots , 7\}$ is fixed. Note that
\begin{align*}
&e^{k}\wedge ( -\delta_{k0}e^{01}
-\delta_{k2}e^{03}+\delta_{k3}e^{02}-\delta_{k4}e^{05}
+\delta_{k5}e^{04}-\delta_{k6}e^{07} +\delta_{k7}e^{06}))\\
&=2(e^{023}+ e^{045} +e^{067}).
\end{align*}
For any $1\leq k\leq 7$ define a $2$-form by
$$
\beta_{k}(X, Y) = \langle \phi , i_{X}(i_{e_{k}}\alpha_{0})\wedge
Y - i_{Y}(i_{e_{k}}\alpha_{0})\wedge X\rangle .
$$
From a long but straightforward computation which uses the
expressions of $\alpha_{0}$ and $\phi$ in the frame $\{ e_{0},
\cdots , e_{7}\}$,
\begin{equation}\label{betak}
e^{k}\wedge\beta_{k}= 6(-e^{247} +e^{357}-e^{256} -e^{346})
+8(e^{023} +e^{045}+ e^{067}).
\end{equation}
Combining (\ref{betak}) with the expression of $p(e^{0}\wedge
i_{e_{k}}\alpha )$ we get our claim.
\end{proof}

The following Lemma concludes the proof of Proposition
\ref{a-spin}.

\begin{lem}\label{FINAL} The value of $P\circ {\mathcal T}_{4}$ on $\eta_{0}$
has the following expression:
\begin{equation}\label{fin}
5(P\circ {\mathcal T}_{4})(\eta_{0}) =
-24(e^{023}+e^{045}+e^{067}) +12(e^{247}-e^{357}+e^{256}+e^{346}).
\end{equation}
In particular, $(P_{8}\circ {\mathcal T}_{4})(\eta_{0})$ and
$(P_{48}\circ {\mathcal T}_{4})(\eta_{0})$ are non-zero.
\end{lem}

\begin{proof} Relation (\ref{fin}) follows from relation (\ref{pc}) and the previous Lemmas.
A direct check shows that (\ref{fin}) is not of the form
$i_{X}\psi$, for $X\in TM$  and thus $(P_{8}\circ {\mathcal
T}_{48}) (\eta_{0})$ is non-zero. Moreover it can be checked that
$$
5(P\circ {\mathcal T}_{4})(\eta_{0})\wedge \psi = -24 e^{0234567}.
$$
In particular,  $(P_{8}\circ {\mathcal T})(\eta_{0})$ is also
non-zero.
\end{proof}

The proof of Proposition \ref{a-spin} is now completed.
Proposition \ref{p-spin} and Theorem \ref{main} {\it ii)} follow.

LIANA DAVID: Institute of Mathematics Simion Stoilow of the
Romanian Academy, Calea Grivitei no. 21, Sector 1, Bucharest,
Romania; e-mail: liana.david@imar.ro

\end{document}